\documentclass[reqno,12pt]{amsart} %
\usepackage{amsmath,amstext, amsthm, amssymb,amsfonts}
\usepackage{fancybox}

 \usepackage[top=1 in,bottom=1in, left=1 in, right= 1 in]{geometry}

\numberwithin{equation}{section}

\newcommand{\calK}{\mathcal{K}}

\newtheorem{theorem}{Theorem}[section]
\newtheorem{proposition}[theorem]{Proposition}
\newtheorem{lemma}[theorem]{Lemma}

\theoremstyle{definition}
\newtheorem{definition}[theorem]{Definition}
\newtheorem{remark}[theorem]{Remark}
\newtheorem{example}[theorem]{Example}

\newcount\mins \newcount\hours \hours=\time \mins=\time
\divide\hours by60 \multiply\hours by 60 \advance\mins by-\hours
\divide\hours by60

\def\now{ \ifnum\hours>11 \ifnum\hours>12 \advance\hours by
-12 \fi \number\hours:\ifnum\mins<10 0\fi \number\mins\ pm,\ \else
\ifnum\hours=0 \hours=12 \fi \number\hours:\ifnum\mins<10 0\fi
\number\mins\ am,\ \fi}

\newcommand{\V}{\mathbb{V}}

\newcommand{\bea}{\begin{equation}}
\newcommand{\eea}{\end{equation}}

\newcommand{\CSK}{Cauchy-Stieltjes Kernel (CSK)\ \renewcommand{\CSK}{CSK\ }}
\newcommand{\FEF}{Free Exponential  (FE)\ \renewcommand{\FEF}{FE\ }}

\title{On Cauchy-Stieltjes Kernel Families}
\author{
W{\l}odek  Bryc}\thanks{\noindent Research partially supported by
NSF grant  DMS-0904720
 and by the Taft Research Center}

\date{ Created   April 29, 2010.\\
{\em File:} {\tt \jobname .tex} of  \today, \now }

\address{
Department of Mathematical Sciences,\\
University of Cincinnati,\\
PO Box 210025,\\
Cincinnati, OH 45221--0025, USA} \email{brycw@math.uc.edu}

\author{Raouf Fakhfakh}

\address{Faculty of Sciences\\
Sfax University \\
PO Box 1171, 3000, Sfax, Tunisia }
\email{abdelhamid.hassairi@fss.rnu.tn}

\author{Abdelhamid Hassairi}
\address{Faculty of Sciences\\
Sfax University \\
PO Box 1171, 3000, Sfax, Tunisia }




\begin{document}


\begin{abstract}
We explore properties of Cauchy-Stieltjes families that have no counterpart in exponential families.
We  relate the variance function of the iterated  Cauchy-Stieltjes  family to the pseudo-variance function of the initial   Cauchy-Stieltjes family. We  also investigate when the domain of means can be extended  beyond the "natural domain".
\end{abstract}

\maketitle

%
%
%

\section{Introduction}
This paper is a continuation of the study of \CSK families. Our goal
here is to advance the understanding of two phenomena that have no
known analogues for the classical exponential families. Firstly, a
typical member of a given \CSK family generates a different \CSK
family, so one can construct new \CSK families by the iteration
process. Secondly, in the natural parametrization of a \CSK family
by the mean, one can sometimes  extend the family beyond the natural
domain of means, preserving the variance function when the variance exists.

The notations used in what follows are the ones used in
Bryc-Hassairi \cite{Bryc-Hassairi-09}.
Throughout the paper $\nu$ is a non-degenerate probability measure
with support bounded from above. Then
\begin{equation}   \label{M(theta)}
M(\theta)=\int \frac{1}{1- \theta x} \nu(dx)
\end{equation}
 is well defined for all $\theta\in [0,\theta_+)$ with $1/\theta_+=\max\{0, \sup {\rm supp} (\nu)\}$ and
 \begin{equation}\label{K_and_F}
\mathcal{K}_+(\nu)=\{P_\theta(dx); \theta\in(0,\theta_+)\}=\{Q_m(dx), m\in(m_0,m_+)\}
 \end{equation}
is the \CSK family generated by $\nu$. That is,
$$P_\theta(dx)=\frac{1}{M(\theta)(1-\theta
x)}\nu(dx)$$ and  $Q_m(dx)$ is the corresponding parametrization by
the mean, which for $m\ne0$ is given by
\begin{equation} \label{F(V)}
 Q_m(dx)=\frac{\V(m)}{\V(m)+m(m-x)}\nu(dx),
\end{equation} 
with
$$ Q_0(dx)=\frac{\V'(0)}{\V'(0)-x}\nu(dx)$$ if $m_0<0<m_+$, and  which involves
the {\em pseudo-variance function} $\V(m)$. The interval $(m_0,m_+)$
is called the (one sided) domain of means, and is determined as the
image of $(0,\theta_+)$ under the strictly increasing function
$k(\theta)=\int x P_\theta(dx)$ which is   given by the formula
\begin{equation}\label{L2m}
k(\theta)=\frac{M(\theta)-1}{\theta M(\theta)}.
\end{equation}

The pseudo-variance function has straightforward probabilistic interpretation if
$m_0=\int x d\nu$ is finite. Then, see \cite[Proposition
3.2]{Bryc-Hassairi-09} we know that
\begin{equation}\label{VV2V}
\frac{\V(m)}{m}=\frac{v(m)}{m-m_0}.\end{equation} where
\begin{equation}
  \label{Def Var}
  v(m)=\int (x-m)^2 Q_m(dx)
\end{equation}
is the so called variance function of the family $\{Q_m\}$. In
particular, $\V=v$ when $m_0=0$.

In general,
\begin{equation}\label{m2v}
\frac{\V(m)}{m}=\frac{1}{\psi(m)}-m,
\end{equation}
where $\psi:(m_0,m_+)\to (0,\theta_+)$ is the inverse of the
function $k(\cdot)$.
 From \eqref{m2v}  it is clear that when
$0\in(m_0,m_+)$, we must have $\V(0)=0$. In this case, we assign
the value $1/\psi(0)$ to the undefined expression $\V(m)/m$ at
$m=0$.

The generating measure $\nu$ is determined uniquely by the
 pseudo-variance function $\V$ through the following identities (for technical details, see \cite{Bryc-Hassairi-09}):
if
\begin{equation}\label{z2m}
 z=z(m)=m+\frac{\V(m)}{ m}
 \end{equation}
 then the Cauchy transform
\begin{equation}\label{G-transform}
G_\nu(z)=\int\frac{1}{z-x}\nu(dx).
\end{equation}
satisfies
 \begin{equation}
  \label{G2V}
  G_\nu(z)=\frac{m}{\V(m)}.
\end{equation}

%
%
%

Let
\begin{equation}\label{B(nu)}
A=A(\nu)=\sup supp(\nu),\ \ \ B=B(\nu)=\max\{0,\ A(\nu)\}.
\end{equation}
We note that  $B(\nu)=1/\theta_+\in[0,\infty)$.

From \cite[Remark 3.3]{Bryc-Hassairi-09}  we read out the following.
\begin{proposition}[\cite{Bryc-Hassairi-09}]\label{P-BH}
For a non-degenerate probability measure $\nu$ with support bounded
from above,  the one-sided domain of means $(m_0,m_+)$ of  is
determined from the following formulas
\begin{equation}\label{m0}
m_0=\lim_{\theta\to 0+} k(\theta)
\end{equation}
and with $B=B(\nu)$,
\begin{equation}\label{m+}
m_+=B-\lim_{b\to B+}\frac{1}{G(b)}.
\end{equation}
\end{proposition}

\begin{remark}\label{R1.3}We list some additional properties relevant to this paper. \begin{itemize}
\item[(i)]  $m_0<m_+$
\item[(ii)]
Since $1/\psi(m)=m+\V(m)/m$, we know that
$$m+\V(m)/m>m_++\V(m_+)/m_+=1/\theta_+=B(\nu)\geq 0.$$ In particular, $G(m+\V(m)/m)$ is well defined, and non-negative.
\item [(iii)]Since $v(m)\geq 0$, from \eqref{VV2V} we see that $\V(m)/m>0$ for
$m\in(m_0,m_+)$. This holds also when the variance is infinite --
just apply Remark \ref{R1.3} and \eqref{G2V}.
\end{itemize}

\end{remark}

 \section{Iterated \CSK families}
 One difference between the exponential and \CSK families is that one can build nontrivial  iterated  \CSK families.  That is, each member of an exponential family generates the same exponential family so it does not matter which of them we use for the generating measure. But this is not so for \CSK families: each member of a \CSK family generates something different than the original family, so the construction can be iterated.

 Suppose $Q_m$ is in the \CSK family generated by a probability measure $\nu$ with support bounded, say, from above, as given by \eqref{F(V)}. Then necessarily $Q_m$ has the support bounded from above, but it also has one more moment than $\nu$.  Consider now a new \CSK family generated by $Q_m$. Then, as long as $m\ne m_0$,  the variance function of this new family necessarily exists. Our goal is to relate the variance function of this new family to the pseudo-variance function of the initial family.


%

 \subsection{Example: iterated  semicircle \CSK families} Here we use integral identities related to the semicircle law to construct iterated \CSK families by elementary means. The iterations get progressively more cumbersome, and illustrate the need for the general theory.

 For complex $a_1,a_2,a_3,a_4$ let
$$
\widetilde{f}(x;a_1,a_2,a_3,a_4)=\sqrt{4-x^2}\prod_{j=1}^4 (1+a_j^2-a_jx)^{-1} $$

Our starting point is the following formula.
\begin{lemma}\label{L3.1}
If $|a_1|,\dots,|a_4|<1$, then
\begin{equation}\label{AW_int}
\int_{-2}^2\:\widetilde{f}(x;a_1,a_2,a_3,a_4)\:dx=K(a_1,a_2,a_3,a_4)\;,
\end{equation}
where \begin{equation}\label{<1} K(a_1,a_2,a_3,a_4)=2\pi
(1-a_1a_2a_3a_4) \prod_{1\leq i<j\leq 4}(1-a_ia_j)^{-1}\;.
\end{equation}
\end{lemma}
Applying this formula with $a_1=a_2=a_3=0$ and
$a_4=m\in(0,1)$ (recall that $m_0=0$ and we are in one-sided setting as in \eqref{K_and_F}), we
get
$$\int_{-2}^2 \frac{\sqrt{4-x^2}}{1+m(m-x)}dx=2\pi,$$ so
$$Q_m=\frac{\sqrt{4-x^2}}{2\pi(1+m(m-x))}1_{|x|<2}dx.$$

Now  we fix $Q_{m_1}\in\calK_+(\nu)$ with mean $m_1$. Applying
\eqref{AW_int}  with $a_1=a_2=0$ $a_3=a$ and $a_4=m_1$, we get
$$\int_{-2}^2 \frac{\sqrt{4-x^2}}{(1+a(a-x))((1+m_1(m_1-x)))}dx=\frac{2\pi}{1-am_1},$$
i.e.
$$\int_{-2}^2 \frac{1-am_1}{1+a(a-x)}Q_{m_1}(dx)=1.$$
Rewriting this into the form suggested by \eqref{F(V)}, we see that
\begin{multline}
\frac{1-am_1}{1+a(a-x)}=\frac{1}{\frac{1+a^2}{1-am_1} -
\frac{a}{1-am_1}
x}=\frac{1}{1+\frac{a(a+m_1)}{1-am_1}-\frac{a}{1-am_1}x}
\\=\frac{1}{1+\frac{a}{1-am_1}(a+m_1-x)}=\frac{1}{1+\frac{a(a+m_1)}{(1-am_1)(a+m_1)}(a+m_1-x)}
\end{multline}
Taking $m=a+m_1$, from  \eqref{F(V)} we see that the pseudo-variance
function of the \CSK family $\calK_+(Q_{m_1})$ is
$$\V_1(m)=\frac{(1-am_1)(a+m_1)}{a}=\left(1-(m-m_1)m_1\right)\frac{m}{m-m_1},$$
 i.e.    the corresponding variance function
 $$v_1(m)=\frac{m-m_1}{m}\V_1(m)=1+m_1^2- m_1m$$  is an affine function of  $m$.
 It is clear that the formula works for all
 $m\in(m_1,m_1+1)$ (again, we use the one-sided setup   as in \eqref{K_and_F}).
 This variance function corresponds to an affine transformation of the Marchenko-Pastur law, see \cite[Example 4.1]{Bryc-06-08}.
We will see that the same result will follow from general theory, see
\eqref{v1_semicircle}.

 Now  we iterate this procedure. Fix
 $$Q_{m_2,m_1}=\frac{1-m_1(m_2-m_1)}{1+(m_2-m_1)(m_2-m_1-x)}Q_{m_1}(dx) \in\calK_+(Q_{m_1})$$ with mean $m_2$.
 Applying \eqref{AW_int} again,  with $a_1=0$, $a_2=a$ $a_3=m_2-m_1$ and $a_4=m_1$, we get
\begin{multline}
\int_{-2}^2
\frac{\sqrt{4-x^2}}{(1+a(a-x))((1+(m_2-m_1)(m_2-m_1-x)))((1+m_1(m_1-x)))}dx\\=\frac{2\pi}{(1-(m_2-m_1)m_1)(1-am_1)(1-a(m_2-m_1))},
\end{multline}
i.e. 
 $$\int_{-2}^2
\frac{(1-am_1)(1-a(m_2-m_1))}{1+a(a-x)}Q_{m_2,m_1}(dx)=1.$$ As
previously, after the appropriate choice of $a$ we want to represent
 $$\frac{(1-am_1)(1-a(m_2-m_1))}{1+a(a-x)}$$ as \eqref{F(V)}. From
 \begin{multline}
   \frac{(1-am_1)(1-a(m_2-m_1))}{1+a(a-x)}=\frac{(1-am_1)(1-a(m_2-m_1))}{1+a^2 -ax}
   \\=
   \frac{1}{1+\frac{a}{(1-am_1)(1-a(m_2-m_1))}(a (m_1^2-m_2 m_1+1)+m_2-x)}
 \end{multline}
we read out that with $$a=\frac{m-m_2}{m_1^2-m_2 m_1+1}$$ the
pseudo-variance function of the \CSK family $\calK_+(Q_{m_2,m_1})$
is  
\begin{multline}
  \V_2(m)=\frac{(1-am_1)(1-a(m_2-m_1))(a m_1^2-a m_2 m_1+a+m_2)} { a  }
  \\=\frac{m \left(1-\left(m-m_1\right) m_1\right) \left(\left(m_1-m_2\right)
   \left(m+m_1-m_2\right)+1\right)}{\left(m-m_2\right) \left(m_1^2-m_2
   m_1+1\right)}.\end{multline}
So by \eqref{VV2V} the corresponding variance function 
$$v_2(m)=
\frac{\left(1-\left(m-m_1\right) m_1\right) \left(\left(m_1-m_2\right)
   \left(m+m_1-m_2\right)+1\right)}{\left(m_1^2-m_2
   m_1+1\right)}$$ is a quadratic polynomial in $m$.
   The above argument works when  $|a|<1$, i.e. since we are in one-sided setting as in \eqref{K_and_F}, for all $m_2<m<m_2+m_1^2-m_1m_2+1$.
(This variance function corresponds to an affine transformation of
the free Meixner law.)

The calculations for the next iteration that would start with
$Q_{m_3,m_2,m_1}\in\calK_+(Q_{m_2,m_1})$ with mean $m_3$, seems to be too
cumbersome.

\subsection{General approach}
In this section, we show how to  relate to the domains of means and
the  pseudo-variance functions of the original family $\calK_+(\nu)$
and the new family $\calK_+(Q_{m_1})$.

Fix $m_1\in(m_0,m_+)$, and consider $Q_{m_1}
=P_{\theta_1}\in\mathcal{K}_+(\nu)$, with
$\theta_1\in(0,\theta_+)$.\\
Define
$$M_1(\theta)=\int\frac{1}{1-\theta x}P_{\theta_1}(dx),$$ for 
$\theta\in\Theta=\{\theta \geq 0 ; \ M_1(\theta)<\infty\}$.

The \CSK family generated by $Q_{m_1} =P_{\theta_1}$ is

$${\mathcal{K}}_+(P_{\theta_1})=\{\overline{P}_\theta(dx)\}=\left\{\frac{1}{M_1(\theta)(1-\theta x)}P_{\theta_1}(dx),\;\theta\in\Theta\right\} .$$

\begin{proposition}
\begin{enumerate}
  \item     $\Theta=(0,\theta_+)$\\

\item For $\theta\in\Theta$, we have
\begin{equation}
  \label{(5.1)}  M_1(\theta)=
  \begin{cases}
 \displaystyle\frac{\theta M(\theta)-\theta_1 M(\theta_1)}{M(\theta_1)(\theta-\theta_1)} &\mbox{  if $\theta\neq\theta_1$}; \\
 \\
    \displaystyle\frac{M(\theta_1)+\theta_1 M'(\theta_1)}{M(\theta_1)}&\mbox{ if $\theta=\theta_1$}.
                                                    \end{cases}
\end{equation}

\item  For $\theta\in\Theta$, we set $k(\theta)=\displaystyle\int x
P_\theta(dx) $ the mean of $P_\theta$, and
$k_1(\theta)=\displaystyle\int x \overline{P}_\theta(dx) $, the mean
of $\overline{P}_\theta$. Then

\begin{equation}
  \label{(5.2)}
  k_1(\theta)=\begin{cases}
     \displaystyle\frac{\theta k(\theta)-\theta_1
k(\theta_1)}{(\theta-\theta_1)+\theta\theta_1(k(\theta)-k(\theta_1))} &
\mbox{ if $\theta\neq\theta_1$}; \\
\\
\displaystyle\frac{k(\theta_1)+\theta_1 k'(\theta_1)}
            {1+\theta_1^2 k'(\theta_1)}&\mbox{ if $\theta=\theta_1  $}.
    \end{cases}
\end{equation}

\end{enumerate}

\end{proposition}

 \begin{proof}

$(i)$ We have that $$M_1(\theta)=\displaystyle\int\frac{1}{1-\theta
x}P_{\theta_1}(dx)=\displaystyle\int \frac{1}{M(\theta_1)(1-\theta_1
x)(1-\theta x)}\nu(dx) .$$ As $\theta_1\in(0,\theta_+)$, the
function $x\mapsto\frac{1}{1-\theta_1 x} $ is bounded on the support
of $\nu$, so that $M_1(\theta)$ exists for $\theta$ such that the
integral $\displaystyle\int \frac{1}{1-\theta x}\nu(dx)$ converges
that is for $\theta$ in $(0,\theta_+).$

$(ii)$
$$M_1(\theta)=\displaystyle\int\frac{1}{1-\theta x}P_{\theta_1}(dx)=\displaystyle\int\frac{1}{M_1(\theta_1)(1-\theta_1 x)(1-\theta x)}\nu(dx),$$

If $\theta\neq\theta_1$, then,
$$\displaystyle\frac{1}{(1-\theta x)(1-\theta_1
x)}=\displaystyle\frac{\theta}{(\theta-\theta_1)(1-\theta
x)}-\displaystyle\frac{\theta_1}{(\theta-\theta_1)(1-\theta_1 x)}
.$$ It follows that
\begin{eqnarray*}
  M_1(\theta) & = & \displaystyle\frac{\theta}{M(\theta_1)(\theta-\theta_1)}\displaystyle\int\frac{1}{1-\theta x}\nu(dx)-\displaystyle\frac{\theta_1}{M(\theta_1)(\theta-\theta_1)}
  \displaystyle\int\frac{1}{1-\theta_1 x}\nu(dx)\\
& = & \displaystyle\frac{\theta M(\theta)-\theta_1
M(\theta_1)}{M(\theta_1)(\theta-\theta_1)} .
\end{eqnarray*}
If $\theta=\theta_1$, then
$$M_1(\theta_1)=\displaystyle\int\frac{1}{M(\theta_1)(1-\theta_1
x)^2}\nu(dx)=\frac{M(\theta_1)+\theta_1 M'(\theta_1)}{M(\theta_1)}
.$$

$(iii)$   We have that
$$k_1(\theta)=\displaystyle\frac{M_1(\theta)-1}{\theta M_1(\theta)}
.$$ If  $\theta\neq\theta_1$, then
\begin{eqnarray*}
k_1(\theta)  = \displaystyle\frac{\displaystyle\frac{\theta
M(\theta)-\theta_1
M(\theta_1)}{M(\theta_1)(\theta-\theta_1)}-1}{\theta\displaystyle\frac{\theta
M(\theta)-\theta_1 M(\theta_1)}{M(\theta_1)(\theta-\theta_1)}}=
\displaystyle\frac{M(\theta)-M(\theta_1)}{\theta M(\theta)-\theta_1
M(\theta_1)} .
\end{eqnarray*}
As $M(\theta)=\displaystyle\frac{1}{1-\theta k(\theta)},$ we obtain
that
$$k_1(\theta)=\displaystyle\frac{\theta k(\theta)-\theta_1
k(\theta_1)}{(\theta-\theta_1)+\theta\theta_1(k(\theta)-k(\theta_1))}.$$
If  $\theta=\theta_1$, then
$$k_1(\theta_1)=\displaystyle\frac{M_1(\theta_1)-1}{\theta_1M_1(\theta_1)}
=\displaystyle\frac{\displaystyle\frac{(M(\theta_1)+\theta_1
M'(\theta_1))}{M(\theta_1)}-1}{\theta_1\displaystyle\frac{(M(\theta_1+\theta_1
M'(\theta_1)))}{M(\theta_1)}}
 =\displaystyle\frac{M'(\theta_1)}{M(\theta_1)+\theta_1
 M'(\theta_1)}\,.$$
Given that $$M'(\theta_1)=\left.\left(\displaystyle\frac{1}{1-\theta
k(\theta)}\right)'\right|_{\theta=\theta_1}=\displaystyle\frac{k(\theta_1)+\theta_1
k'(\theta_1)}{(1-\theta_1 k(\theta_1))^2} \;,$$ we obtain
$$k_1(\theta_1)=\displaystyle\frac{k(\theta_1)+\theta_1
k'(\theta_1)}{1+\theta_1^2 k'(\theta_1)} .$$
 \end{proof}
Next we denote by $D_+(\nu)$ and $\mathbb{V}$ the domain of the
means and the pseudo-variance function of the family
${\mathcal{K}}_+(\nu)$, and by $D_+(Q_{m_1})$ and $\mathbb{V}_{1}$
the domain of the means and the pseudo-variance function of
${\mathcal{K}}_+(Q_{m_1})$. Recall that $D_+(\nu)=k((0,\theta_+))$
and $D_+(Q_{m_1})=k_1((0,\theta_+))$. We set $$m=k(\theta)\ \
\textrm{and}\ \ \overline{m}=k_1(\theta).$$ We will also use the
inverse $\psi$ of the function $\theta\longmapsto k(\theta)$ from
$(0,\theta_+)$ into $(m_0,m_+)$, and the inverse $\psi_1$ of the
function $\theta\longmapsto k_1(\theta)$ from $(0,\theta_+)$ into
its image $(\overline{m}_0,\overline{m}_+)$.

\begin{theorem}\label{T1}
Let $\nu$ be a probability measure with support bounded from above,
and let ${\mathcal{K}}_+(\nu)$ be the CSK family generated by $\nu$.
Fix $m_1\in(m_0,m_+)$ and let $B=B(\nu)$ be given by \eqref{B(nu)}.
With the notations introduced above, we have
\begin{enumerate}
\item
\begin{equation}
  \label{(5.3)}
  \ \   \overline{m}=k_1(\psi(m))=\begin{cases}
        \displaystyle\frac{m^2 \mathbb{V}(m_1)-m_1^2 \mathbb{V}(m)}{m\mathbb{V}(m_1)-m_1\mathbb{V}(m)} & \mbox{ if $ m\neq m_1 $}\\ \\
        \displaystyle\frac{2m_1 \mathbb{V}(m_1)-m_1^2\mathbb{V}'(m_1)}{\mathbb{V}(m_1)-m_1\mathbb{V}'(m_1)} &\mbox{ if $ m=m_1$}.
                                    \end{cases}
\end{equation}
\item the (one sided) domain of means is
$$D_+(Q_{m_1})=(\overline{m}_0,\overline{m}_+)=\left(m_1,\displaystyle\frac{m_+G_\nu(B)-
m_1^2/\mathbb{V}(m_1)}{G_\nu(B)-m_1/\mathbb{V}(m_1)}\right).$$
(interpreted as the limit $b\to B^+$.)
\item
\begin{equation}
  \label{(5.4)} \displaystyle\frac{\mathbb{V}_1(\overline{m})}{\overline{m}}+\overline{m}=\displaystyle\frac{\mathbb{V}(m)}{m}+m .
\end{equation}

\end{enumerate}

\end{theorem}

Note that the function $m\longmapsto\overline{m}$ is a bijection
from $D_+(\nu)$ into $D_+(Q_{m_1})$, so that to get explicitly the
pseudo-variance function of the CSK family
${\mathcal{K}}_+(Q_{m_1})$, we need to express $m$ in terms of
$\overline{m}$ from \eqref{(5.3)} and insert it in \eqref{(5.4)}.

 \begin{proof}
 $(i)$ Suppose that $m\neq m_1$.
\begin{eqnarray*}
\overline{m} & = & k_1(\psi(m)) \\
& = & \displaystyle\frac{m\psi(m)-m_1\psi(m_1)}{(\psi(m)-\psi(m_1))+\psi(m)\psi(m_1)(m-m_1)} \\
& = & \displaystyle\frac{m^2(\mathbb{V}(m_1)+m_1^2)-m_1^2(\mathbb{V}(m)+m^2)}{m(\mathbb{V}(m_1)+m_1^2)-m_1(\mathbb{V}(m)+m^2)+m m_1(m-m_1)}\\
& = & \displaystyle\frac{m^2 \mathbb{V}(m_1)-m_1^2 \mathbb{V}(m)}{m
\mathbb{V}(m_1)-m_1 \mathbb{V}(m)} .
\end{eqnarray*}
For $m=m_1$, we have
\begin{eqnarray*}
\overline{m}_1=k_1(\psi(m_1))& = &k_1(\theta_1)  =
\displaystyle\frac{k(\theta_1)+\theta_1 k'(\theta_1)}{1+\theta_1^2
k'(\theta_1)}
 =  \displaystyle\lim_{\theta\longrightarrow\theta_1}\displaystyle\frac{\theta k(\theta)-\theta_1 k(\theta_1)}{(\theta-\theta_1)+\theta\theta_1 (k(\theta)-k(\theta_1))}\\
& =& \displaystyle\lim_{m\longrightarrow m_1}\displaystyle\frac{m^2 \mathbb{V}(m_1)-m_1^2 \mathbb{V}(m)}{m \mathbb{V}(m_1)-m_1 \mathbb{V}(m)}\\
& = &  \displaystyle\lim_{m\longrightarrow
m_1}\displaystyle\frac{\mathbb{V}(m_1)\mathbb{V}(m)(\displaystyle\frac{m^2}{\mathbb{V}(m)}-
\displaystyle\frac{m_1^2}{\mathbb{V}(m_1)})}{\mathbb{V}(m_1)\mathbb{V}(m)(\displaystyle\frac{m}{\mathbb{V}(m)}-
\displaystyle\frac{m_1}{\mathbb{V}(m_1)})}\\
& = & \displaystyle\frac{(m^2/\mathbb{V}(m))'}{(m/\mathbb{V}(m))'}\mid_{m=m_1}\\
& = & \displaystyle\frac{2m_1
\mathbb{V}(m_1)-m_1^2\mathbb{V}'(m_1)}{\mathbb{V}(m_1)-m_1\mathbb{V}'(m_1)}
.
\end{eqnarray*}
$(ii)$ Using the definition of the domain of means,
\begin{eqnarray*}
 \overline{m}_0=\displaystyle\lim_{\theta\longrightarrow
0}k_1(\theta)  =  \displaystyle\lim_{m\longrightarrow
m_0}\displaystyle\frac{m^2 \mathbb{V}(m_1)-m_1^2
\mathbb{V}(m)}{m\mathbb{V}(m_1)-m_1\mathbb{V}(m)} =
\displaystyle\lim_{m\longrightarrow
m_0}\displaystyle\frac{\displaystyle\frac{m^2}{\mathbb{V}(m)}-\displaystyle\frac{m_1^2}
{\mathbb{V}(m_1)}}{\displaystyle\frac{m}{\mathbb{V}(m)}-\displaystyle\frac{m_1}{\mathbb{V}(m_1)}}=m_1
.
\end{eqnarray*}
\begin{eqnarray*}
\overline{m}_+=\displaystyle\lim_{\theta\longrightarrow\theta_+}k_1(\theta)
& = & \displaystyle\lim_{m\longrightarrow m_+}\displaystyle\frac{m^2
\mathbb{V}(m_1)-m_1^2
\mathbb{V}(m)}{m\mathbb{V}(m_1)-m_1\mathbb{V}(m)}
 =  \displaystyle\lim_{m\longrightarrow m_+}\displaystyle\frac{\displaystyle\frac{m^2}{\mathbb{V}(m)}-\displaystyle\frac{m_1^2}{\mathbb{V}(m_1)}}
{\displaystyle\frac{m}{\mathbb{V}(m)}-\displaystyle\frac{m_1}{\mathbb{V}(m_1)}}\\
& = &
\lim_{b\to B^+}\displaystyle\frac{m_+G_\nu(b)-\displaystyle\frac{m_1^2}{\mathbb{V}(m_1)}}{G_\nu(b)-\displaystyle\frac{m_1}{\mathbb{V}(m_1)}}
=
\displaystyle\frac{m_+G_\nu(B)-\displaystyle\frac{m_1^2}{\mathbb{V}(m_1)}}{G_\nu(B)-\displaystyle\frac{m_1}{\mathbb{V}(m_1)}}
.
\end{eqnarray*}
(This is  $m_+$ when $\lim_{b\to B^+}G(b)=\infty$.)

$(iii) $   For $\theta\in(0,\theta_+)$ we have $
\theta=\psi(k(\theta))=\psi_1(k_1(\theta))$ so that $
\psi(m)=\psi_1(\overline{m})\ .$ By  \eqref{m2v}, this implies
\eqref{(5.4)}.
\end{proof}
Note that as the probability measure $Q_{m_1}$ has a finite first
moment $\overline{m}_0=m_1$, the variance function $v_1(.)$ of the
CSK family ${\mathcal{K}}_+(Q_{m_1})$ exists and from \eqref{VV2V}
we have
 $$\V_1(\overline{m})=\displaystyle\frac{\overline{m}}{\overline{m}-m_1}v_1(\overline{m}) .$$

\subsection{Applications}
The following examples illustrate the usefulness of Theorem
\ref{T1}, and  provide
 examples of \CSK families with rational variance functions.
\subsubsection{\CSK families with quadratic variance function}
The \CSK families with quadratic variance function have
\begin{equation}\label{QV}
v(m)= 1+am+bm^2 =\mathbb{V}(m),
\end{equation}
(we consider centered case here, with  $m_0=0$).

Formula \eqref{(5.3)} gives that
$m=\displaystyle\frac{\overline{m}-m_1}{1+am_1+bm_1\overline{m}}.$

Formula \eqref{(5.4)} gives that
$$\mathbb{V}_1(\overline{m})=\displaystyle\frac{\overline{m}}{(\overline{m}-m_1)
(1+am_1+bm_1\overline{m})}P(\overline{m}),$$ {where}
$P(\overline{m})=( 1+ \overline{m} (a+b \overline{m})) \left(1+m_1
   (a-\overline{m}+(b+1)
   m_1)\right)$. \\The corresponding
variance function is
 $$v_1(\overline{m})=\displaystyle\frac{1}{1+am_1+bm_1\overline{m}}P(\overline{m}),$$

The following two special cases are of interest.
\begin{example}
The Wigner's semicircle (free Gaussian) law
$$\nu(dx)=\displaystyle\frac{\sqrt{4-x^2}}{2\pi}1_{(-2,2)}(x)dx ,$$
has a constant variance function i.e. \eqref{QV} holds with $a=b=0$:
 $v(m)=1={\V}(m)$ and the (one-sided) domain of means is $D_+(\nu)=(0,1).$ (The full two-sided domain of means is of course $(-1,1)$.) For $m_1\in D_+(\nu)$, the probability measure

 $$Q_{m_1}(dx)=\displaystyle\frac{\sqrt{4-x^2}}{2\pi(1+m_1(m_1-x))}1_{(-2,2)}(x)dx ,$$
generates \CSK family with pseudo-variance function
 $\mathbb{V}_1(\overline{m})=\displaystyle\frac{\overline{m}}{\overline{m}-m_1}(-m_1\overline{m}+m_1^2+1)$, and domain of means $D_+(Q_{m_1})=(m_1,1+m_1)$. The corresponding variance function is
\begin{equation}\label{v1_semicircle}
v_1(\overline{m})=-m_1\overline{m}+m_1^2+1.
\end{equation}

Up to affine transformation, this is in fact the Marchenko-Pastur
law, see next example.
\end{example}
\begin{example}
 The (absolutely continuous) Marchenko-Pastur (free Poisson) law
 $$\nu(dx)=\displaystyle\frac{\sqrt{4-(x-a)^2}}{2\pi(1+ax)}1_{(a-2,a+2)}(x)dx$$
 corresponds to \eqref{QV} with  $b=0$ and
$0<a^2<1$. The variance function is  $v(m)=1+am={\V}(m)$, and the
domain of means is $D_+(\nu)=(0,1).$\\ For $m_1\in D_+(\nu)$, the
probability measure
$$Q_{m_1}(dx)=\displaystyle\frac{(1+am_1)\sqrt{4-(x-a)^2}}{2\pi(1+m_1(a+m_1-x))(1+ax)}1_{(a-2,a+2)}(x)dx$$
generates \CSK family with pseudo-variance function
$${\V}_1(\overline{m})=\displaystyle\frac{\overline{m}}{(1+am_1)(\overline{m}-m_1)}(1+a \overline{m}) \left(1+m_1 \left(a+m_1-\overline{m}\right)\right).$$
The domain of means is
$$D_+(Q_{m_1})=(m_1,1+(a+1) m_1).$$
The variance function is
$$
v_1(\overline{m})=\frac{(1+a \overline{m}) \left(1+m_1
\left(a+m_1-\overline{m}\right)\right)}{1+a
   m_1}.
$$

%
\end{example}
\begin{example} For $a^2>1$, the Marchenko Pastur law is
$$\nu(dx)=\displaystyle\frac{\sqrt{4-(x-a)^2}}{2\pi(1+ax)}1_{(a-2,a+2)}(x)dx+ (1-1/a^2) \delta_{-1/a}(dx)$$
If $a>1$, $B(\nu)=a+2$ and the upper endpoint of the domain of means
is $m_+=1$. In this case, $D_+(\nu)=(0,1),$ and for $m_1\in
D_+(\nu)$, we have
\begin{multline}
Q_{m_1}(dx)=\frac{(1+am_1)\sqrt{4-(x-a)^2}}{2\pi(1+m_1(a+m_1-x))(1+ax)}1_{(a-2,a+2)}(x)dx\\
+\frac{1+am_1}{1+m_1(a+m_1+1/a)}(1-1/a^2) \delta_{-1/a}(dx).
\end{multline}
 This distribution generates the \CSK family with pseudo-variance function
$${\V}_1(\overline{m})=\displaystyle\frac{\overline{m}}{(1+am_1)(\overline{m}-m_1)}(1+a \overline{m}) \left(1+m_1 \left(a+m_1-\overline{m}\right)\right),$$
and domain of means
$$D_+(Q_{m_1})=(m_1,1+(a+1) m_1).$$
The variance function is
$$
v_1(\overline{m})=\frac{(1+a \overline{m}) \left(1+m_1
\left(a+m_1-\overline{m}\right)\right)}{1+a
   m_1}.
$$
{If $a<-1$, then $B(\nu)=-1/a$}, and the domain of means is
$D_+(\nu)=(0,-1/a)$.
 For
$m_1\in D_+(\nu)$, we have that
\begin{multline}
Q_{m_1}(dx)=\frac{(1+am_1)\sqrt{4-(x-a)^2}}{2\pi(1+m_1(a+m_1-x))(1+ax)}1_{(a-2,a+2)}(x)dx\\+
\frac{1+am_1}{1+m_1(a+m_1-x)}(1-1/a^2) \delta_{-1/a}(dx).
\end{multline}
It generates the \CSK family with pseudo-variance function
$${\V}_1(\overline{m})=\displaystyle\frac{\overline{m}}{(1+am_1)(\overline{m}-m_1)}(1+a \overline{m}) \left(1+m_1 \left(a+m_1-\overline{m}\right)\right),$$
with domain of means
$$D_+(Q_{m_1})=(m_1,-1/a).$$
The variance function is, in this case,
$$
v_1(\overline{m})=\frac{(1+a \overline{m}) \left(1+m_1
\left(a+m_1-\overline{m}\right)\right)}{1+a
   m_1}.
$$
\end{example}

\subsubsection{\CSK families with cubic pseudo-variance function :}

For $a>0$, the cubic pseudo-variance function
\begin{equation}\label{CV}
{\V}(m)=m(am^2+bm+c)
\end{equation}   corresponds to \CSK families without variance.
Formula \eqref{(5.3)} gives that
$m=-\displaystyle\frac{\overline{m}(b+am_1)+c}{a(\overline{m}-m_1)}.$
Formula \eqref{(5.4)} gives
${\V}_1(\overline{m})=\displaystyle\frac{\overline{m}}{a(\overline{m}-m_1)^2}Q(\overline{m})
$, so the corresponding variance function is
$$v_1(\overline{m})=\displaystyle\frac{1}{a(\overline{m}-m_1)}Q(\overline{m}),$$
with
$Q(\overline{m})=(c+\overline{m} (b+a \overline{m}))
\left(c-\overline{m}+m_1 \left(b+a m_1+1\right)\right)$.

%

The following special cases are of interest
\begin{example}\label{Ex:Abel}
The Free Abel (or Free Borel-Tanner) law
$$\nu(dx)=\displaystyle\frac{1}{\pi(1-x)\sqrt{-x}}1_{(-\infty,0)}(x)dx $$
has domain of means $D_+(\nu)=(-\infty,0)$ and pseudo-variance
function $\mathbb{V}(m)=m^2(m-1)$. For $m_1\in D_+(\nu)$,
probability measure
$$Q_{m_1}(dx)=\displaystyle\frac{m_1(m_1-1)}{\pi(m_1^2-x)(1-x)\sqrt{-x}}1_{(-\infty,0)}(x)dx,$$
generates \CSK family with pseudo-variance function
$$ \mathbb{V}_1(\overline{m})=\displaystyle\frac{\overline{m}^2}
{(\overline{m}-m_1)^2}(1-\overline{m})(\overline{m}-m_1^2)\,,
$$
and the domain of means $D_+(Q_{m_1})=(m_1,0)$. The corresponding
variance function is
$$ v_1(\overline{m})=\displaystyle\frac{\overline{m}}
{\overline{m}-m_1}(1-\overline{m})(\overline{m}-m_1^2).$$
%
\end{example}
\begin{example}
The free Ressel (or free Kendall) law
$$\nu(dx)=\displaystyle\frac{-1}{\pi x\sqrt{-1-x}}1_{(-\infty,-1)}(x)dx $$
has domain of means  $D_+(\nu)=(-\infty,-2)$ and the pseudo-variance
function $\mathbb{V}(m)=m^2(m+1)$. For $m_1\in D_+(\nu)$, the
probability measure
$$Q_{m_1}(dx)=\displaystyle\frac{-m_1(1+m_1)}{\pi x(m_1^2+2m_1-x)\sqrt{-1-x}}1_{(-\infty,-1)}(x)dx, $$
generates \CSK family with pseudo-variance function
$$
  \mathbb{V}_1(\overline{m})=\displaystyle\frac{\overline{m}^2}{(\overline{m}-m_1)^2}
\left(-\overline{m}^2+(m_1^2+2m_1-1)\overline{m}+m_1^2+2m_1\right).
$$
and the domain of means
$$ D(Q_{m_1})=\left(m_1,\displaystyle\frac{2m_1}{1-m_1}\right).$$
The corresponding variance function is
$$v_1(\overline{m})=\displaystyle\frac{\overline{m}}{\overline{m}-m_1}
(\overline{m}+1)(m_1^2+2m_1-\overline{m}).$$

\end{example}
\begin{example} The free strict arcsine law
$$\nu(dx)=\displaystyle\frac{\sqrt{3-4x}}{2\pi(1+x^2)}1_{(-\infty,3/4)}(x)dx $$
has pseudo-variance function $\mathbb{V}(m)=m(1+m^2)$, and the
domain of means $ D_+(\nu)=(-\infty,-1/2)$. For $m_1\in D_+(\nu)$,
probability measure
$$Q_{m_1}(dx)=\displaystyle\frac{(m_1^2+1)\sqrt{3-4x}}{2\pi(m_1^2+m_1+1-x)(1+x^2)}1_{(-\infty,3/4)}(x)dx$$
generates \CSK family with pseudo-variance function
$$\mathbb{V}_1(\overline{m})=\displaystyle\frac{\overline{m}}
{(\overline{m}-m_1)^2}(-\overline{m}^3+(m_1^2+m_1+1)\overline{m}^2-\overline{m}+(m_1^2+m_1+1))$$
and with domain of means
$$D_+(Q_{m_1})=\left(m_1,\displaystyle\frac{2+m_1}{1-2m_1}\right).$$
The corresponding variance function is
$$v_1(\overline{m})=\displaystyle\frac{1+\overline{m}^2}{\overline{m}-m_1}( m_1^2+m_1+1-\overline{m}).$$
\end{example}
\begin{example} \label{Ex:ISC} The inverse semicircle law
$$\nu(dx)=\displaystyle\frac{p\sqrt{-p^2-4x}}{2\pi x^2}1_{(-\infty,-p^2/4)}(x)dx ,$$
corresponds to \eqref{CV} with $a=1/p^2\ , \ b=c=0$. The
pseudo-variance function is $\mathbb{V}(m)=m^3/p^2$, and the domain
of means is $D_+(\nu)=(-\infty,-p^2)$.

For $m_1\in D_+(\nu)$, probability measure
$$Q_{m_1}(dx)=\displaystyle\frac{pm_1^2\sqrt{-p^2-4x}}{2\pi x^2(m_1^2+p^2(m_1-x)) }1_{(-\infty,-p^2/4)}(x)dx $$
generates \CSK family with pseudo-variance function 
$$\mathbb{V}_1(\overline{m})=\displaystyle\frac{\overline{m}^3}
{(\overline{m}-m_1)^2}(m_1^2/p^2+m_1-\overline{m}),$$
and with domain of means
$$D_+(Q_{m_1})=\left(m_1,\frac{p^2
m_1}{p^2-m_1}\right).$$ The corresponding variance function is
$$v_1(\overline{m})=\displaystyle\frac{\overline{m}^2}{(\overline{m}-m_1)}(m_1^2/p^2+m_1- \overline{m}).$$

%
\end{example}

\section{Extending the domain for parametrization by the mean} \label{Q5.}
 Given a compactly supported measure $\nu$,  Proposition \ref{P-BH} tells us  how to determine  the  one-sided domain of means $(m_0,m_+)$ and how to compute the pseudo-variance function $\V(m)$ for $m\in(m_0,m_+)$.  (There is a similar result for the two-sided domain of means, see \cite[Remark 3.3]{Bryc-Hassairi-09}.)
 But the pseudo-variance function is often well defined for other values of $m$, too. So it is natural to ask whether the corresponding "family of measures" can also be enlarged.  The following example illustrates the idea, drawing on well known properties of the Marchenko-Pastur law.
 \begin{example}\label{Ex:semicircle}
   Consider the (two-sided) \CSK family generated by the semicircle law $\nu=\frac{1}{2\pi}\sqrt{4-x^2}1_{|x|<2}dx$ with the variance function $v(m)=\V(m)=1$, the domain of means $(-1,1)$ and
   $$\mathcal{K}(\nu)=\left\{ \frac{\sqrt{4-x^2}}{2\pi(1+m(m-x))}1_{|x|<2}dx: m\in(-1,1)\right\}.$$
 This is a family of atomless Marchenko-Pastur laws, which can be naturally enlarged to include all Marchenko-Pastur laws:
  $$\overline{\mathcal{K}}(\nu)=\left\{\pi_m(dx)= \frac{\sqrt{4-x^2}}{2\pi(1+m(m-x))}1_{|x|<2}dx+ (1-1/m^2)^+\delta_{m+1/m}:\; m\in(-\infty,\infty)\right\}$$
Noting that
$$\int \pi_m(dx)=1, \; \int x \pi_m(dx)=m,\; \int (x-m)^2 \pi_m(dx)=1,$$
we see that $v(m)=1$ is the variance function of this enlarged
family.
 \end{example}
 Of course, it may also happen that the extension beyond the natural
domain of means is not possible. 
family is full.
\begin{example}\label{Ex:3.2}
Let $\nu=\frac{1}{2}\delta_{-1}+\frac{1}{2}\delta_1$ be the
symmetric Bernoulli distribution. Then
$M(\theta)=\frac{1}{1-\theta^2}$ and $m(\theta)=\theta$. The
(two-sided) range of parameter is $\Theta=(-1,1)$. So the   domain
of means here is $(-1,1)$, and with  $m_0=0$ the pseudo-variance
function is equal to the variance function,
$$
v(m)=\V(m)=1-m^2.
$$
In this case, the variance function is negative outside the domain
of means, so we cannot extend the family $\{Q_m: m\in(-1,1)\}$
beyond the original domain of means while preserving the variance
function $v(m)$, and the relation between $v(m)$ and the
Cauchy-Stieltjes transform.
\end{example}
Our next example shows that the extension sometimes  may proceed in two separate steps.

\begin{example}\label{Ex:3.3}
Consider the inverse semicircle law from Example \ref{Ex:ISC} with
$p=1$. Since $m^2+m\geq -1/4$, it  is clear that  measure $Q_m$ is
non-negative and well defined for all  $m$. Since the integral $\int
Q_m(dx)$ is an analytic function of $m<-1/2$, it must be $1$, so
$Q_m$ is a probability measure for all $m<-1/2$. This is the "first
part" of the extension, from $(-\infty,-1)$  to a larger interval
$(-\infty,-1/2)$.

At $m=-1/2$ the integrand has singularity at
$x=-1/4$ but the integral is still $1$, see the calculation below.
For $m>-1/2$,  the mass becomes less then one,  as $\int
Q_m(dx)=m^2/(1+m)^2$.
So for $m>-1/2$ we can define a new probability measure
 \begin{multline}\label{ISQ}
 \overline Q_m(dx)=Q_m(dx)+\left(1-\frac{m^2}{(1+m)^2}\right)\delta_{m+m^2}(dx)
 = Q_m(dx)+\frac{(1+2m)_+}{(1+m)^2}\delta_{m+m^2}(dx)
 \end{multline}
 with extra mass in the atomic part.

 The definition \eqref{m2v} of pseudo-variance is not directly  applicable beyond $m>-1$. However, if we use relation \eqref{G2V}, then $\V(m)=m^3$ also for $m>-1$. Thus we may claim that the family
$\{\overline Q_m(dx)\}$ extends the domain of means for $\V(m)=m^3$ to $(-\infty,\infty)$.

\end{example}

We now prove the above two claims.
\begin{proof}[Proof of the claims in Example \ref{Ex:3.3}]
 By the change of variable $t=\sqrt{-1-4x}$ in
 $$\displaystyle\int
Q_m(dx)=\displaystyle\int_{-\infty}^{-1/4}\displaystyle\frac{m^2\sqrt{-1-4x}}{2\pi
x^2(m^2+m-x) }dx,$$
 we obtain
 $$\displaystyle\int
Q_m(dx)=\frac{16m^2}{\pi}\displaystyle\int_0^{+\infty}\frac{t^2}{(t^2+1)^2((2m+1)^2+t^2)}dt.$$
The integrand can be decomposed as follows
$$\frac{t^2}{(t^2+1)^2((2m+1)^2+t^2)}=\frac{(2m+1)^2}{((2m+1)^2-1)^2(t^2+1)}-\frac{1}{((2m+1)^2-1)(t^2+1)^2}$$
$$-\frac{(2m+1)^2}{((2m+1)^2-1)^2(t^2+(2m+1)^2)}.$$
For real numbers $a, \ b$, $r\neq0$, we denote
$J_n=\int_a^b\frac{dx}{(x^2+r^2)^n}$. Then we have
$$J_{n+1}=\frac{1}{2nr^2}\left((2n-1)J_n+\left[\frac{x}{(x^2+r^2)^n}\right]_a^b\right).$$
Using this, we get:\\
 For $m=-1/2$,
\begin{eqnarray*}
\displaystyle\int Q_{-1/2}(dx)
 & =
 &\frac{4}{\pi}\displaystyle\int_0^{+\infty}\frac{1}{(t^2+1)^2}dt 
 \\
& =
&\frac{4}{\pi}\left(\frac{1}{2}\left(\frac{\pi}{2}+\left[\tfrac{x}{1+x^2}\right]_0^{+\infty}\right)\right)=1.
\end{eqnarray*}
For $m\neq-1/2$
\begin{eqnarray*}
\displaystyle\int Q_m(dx) & = &
\frac{16m^2}{\pi}\displaystyle\int_0^{+\infty}\frac{t^2}{(t^2+1)^2((2m+1)^2+t^2)}dt\\
& = &
\frac{16m^2}{\pi}\Big(\int\frac{(2m+1)^2}{((2m+1)^2-1)^2(t^2+1)}dt-\int\frac{1}{((2m+1)^2-1)(t^2+1)^2}dt\\
& - & \int\frac{(2m+1)^2}{((2m+1)^2-1)^2(t^2+(2m+1)^2)}dt\Big)\\
& = &
\frac{16m^2}{\pi}\Big(\frac{(2m+1)^2}{((2m+1)^2-1)^2}[\arctan(t)]_0^{+\infty}-\frac{1/2}{(2m+1)^2-1}\left(\frac{\pi}{2}+\left[\tfrac{t}{(t^2+1)^2}\right]_0^{+\infty}\right)\\
& - &
\frac{(2m+1)}{((2m+1)^2-1)^2}\left[\arctan(\tfrac{t}{2m+1})\right]_0^{+\infty}\Big).
\end{eqnarray*}
If $m<-1/2$
\begin{eqnarray*}
\displaystyle\int Q_m(dx) & =
&\frac{16m^2}{\pi}(\frac{(2m+1)^2}{((2m+1)^2-1)^2}\frac{\pi}{2}-\frac{1}{((2m+1)^2-1)}\frac{\pi}{4}\\
& - & \frac{(2m+1)}{((2m+1)-1)^2}(-\frac{\pi}{2}))=1.
\end{eqnarray*}
If $m>-1/2$
\begin{eqnarray*}
\displaystyle\int Q_m(dx) & = &
\frac{16m^2}{\pi}\Big(\frac{(2m+1)^2}{((2m+1)^2-1)^2}\frac{\pi}{2}-\frac{1}{((2m+1)^2-1)}\frac{\pi}{4}\\
& - & \frac{(2m+1)}{((2m+1)-1)^2}\frac{\pi}{2}\Big)=\frac{m^2}{(1+m)^2}.
\end{eqnarray*}

We now verify that the atomic part works as needed.

 By the change of variable $t=\sqrt{-1-4x}$ from \eqref{ISQ} we get
\begin{eqnarray*}
\displaystyle\int x Q_m(dx) & = &
\displaystyle\int_{-\infty}^{-1/4}\displaystyle\frac{m^2\sqrt{-1-4x}}{2\pi
x(m^2+m-x) }dx\\
& = &
-\displaystyle\frac{4m^2}{\pi}\displaystyle\int_0^{+\infty}\frac{t^2}{(t^2+1)((2m+1)^2+t^2)}dt\\
& = &
-\displaystyle\frac{4m^2}{\pi}(-\displaystyle\int_0^{+\infty}\frac{1}{4m(1+m)(t^2+1)}dt+\displaystyle\int_0^{+\infty}\frac{(2m+1)^2}{4(m^2+m)((2m+1)^2+t^2)}dt)\\
 & = &-\displaystyle\frac{4m^2}{\pi}(\frac{-1}{4m(1+m)}[\arctan(t)]_0^{+\infty}+\displaystyle\frac{(2m+1)}{4m(1+m)}[\arctan(\frac{t}{2m+1})]_0^{+\infty})\\
 & = &-\displaystyle\frac{4m^2}{\pi}(\frac{-1}{4m(1+m)}\frac{\pi}{2}+\displaystyle\frac{(2m+1)}{4m(1+m)}\frac{\pi}{2})=-\frac{m^2}{1+m}.
\end{eqnarray*}
So
$$ \int x\overline Q_m(dx)=
-\frac{m^2}{1+m}+\frac{1-2m}{(1+m)^2}m(1+m)=m $$
  as expected.

\end{proof}



We now give a general theory that shows how the two-step extension works.
\subsection{The first extension}

 Suppose that the pseudo-variance
function $\V$ extends as a
real analytic function to $(m_0,+\infty)$. Recall notation \eqref{B(nu)} and define
 \begin{equation}\label{m+(A)}
\mathbf{m}_+(\nu) =\inf\{m>m_0:
m+\frac{\V(m)}{m}=A(\nu)\}.
\end{equation}
From Remark \ref{R1.3} we know that    $\mathbf{m}_+(\nu) \geq m_+$ is well defined. We will verify  that one can use \eqref{F(V)} to extend the domain of means to $(m_0, \mathbf{m}_+(\nu))$, preserving the pseudo-variance function. (The definition \eqref{m2v} of pseudo-variance is not directly  applicable beyond $m>m_+$, so we use an equivalent definition).

 \begin{theorem}\label{T:ext1}
Formula \eqref{F(V)} defines  the  family of probability measures
$\{ Q_m(dx): m\in (m_0,\mathbf{m}_+)\}$, parametrized by the mean $m=\int x Q_m(dx)$. The  Cauchy-Stieltjes transform  of the generating measure $\nu$ satisfies
 \eqref{G2V}  with $z$ given by \eqref{z2m} for all  $m\in (m_0,\mathbf{m}_+)$. In particular, if $\nu$ has finite first moment $m_0$ then  for $m\in (m_0,\mathbf{m}_+)$ the variance of $  Q_m(dx)$ is given by \eqref{VV2V}.
 \end{theorem}

The rest of this section contains proof of Theorem \ref{T:ext1}.

We consider the set $\Theta$ for which the transform \eqref{M(theta)} exists.


In fact, if $A(\nu)\geq0$, then $\Theta=(0,\theta_+)$ with
$\theta_+= \frac{1}{B}$, and if $A(\nu)<0$, then
\begin{equation} \label{BigTheta}
\Theta=\left(-\infty, \frac{1}{A(\nu)}\right)
\cup\left(0,\infty\right).
\end{equation}
One can always write
$$\Theta=\left(0\ ,\displaystyle\frac{1}{B}\right)\cup
\left(\displaystyle\frac{sign(A(\nu))}{B},\displaystyle\frac{1}{A(\nu)}\right)$$
with
$$sign(A(\nu))=\left\{%
\begin{array}{ll}
    1, \ \textrm{if}\ \ A(\nu)\geq0 & \hbox{;} \\
    -1, \ \textrm{if}\ \ A(\nu)<0 & \hbox{.} \\
\end{array}%
\right.   $$
One can then define the first extension of
$\mathcal{K}_+(\nu)$ as
$$\overline{\mathcal{K}}_+(\nu)=\{P_\theta(dx)=\frac{1}{M(\theta)(1-\theta x)}\nu(dx)\ ;\
\theta\in(\displaystyle\frac{sign(A(\nu))}{B},\displaystyle\frac{1}{A(\nu)})\cup(0,\displaystyle\frac{1}{B})\}.$$
Note that $\overline{\mathcal{K}}_+(\nu)={\mathcal{K}}_+(\nu)$
when $A(\nu)\geq0$, because in this case
$\left(\displaystyle\frac{sign(A(\nu))}{B},\displaystyle\frac{1}{A(\nu)}\right)=\emptyset$. Therefore,
the first extension is non-trivial only when $A(\nu)<0$.
\begin{proposition}
Suppose $A(\nu)<0$. For $\theta\in\Theta=\left(-\infty,
\frac{1}{A(\nu)}\right)\cup(0,\infty)$ the mean
\begin{equation}
k(\theta)=\displaystyle\int x
P_\theta(dx)=\frac{M(\theta)-1}{\theta M(\theta)},
\end{equation}
is strictly increasing on $(0,\infty)$ and on
$\left(-\infty,\frac{1}{A(\nu)}\right)$
\end{proposition}
\begin{proof}
It is known  (\cite{Bryc-Hassairi-09}) that the function $k(.)$ is strictly increasing on
$(0,\infty)$, we will use the same reasoning to
show that it is also increasing on
$\left(-\infty, \frac{1}{A(\nu)}\right)$.
 We first observe that for
$\theta\in\left(-\infty,\frac{1}{A(\nu)}\right)$, the expression
$(1-\theta x)$ is negative for all $x$ in the support of $\nu$. In
fact, $x<A(\nu)$ implies that $\theta x>\theta A(\nu)>1$, that is
$1-\theta x<1-\theta A(\nu)<0$. Hence
\begin{eqnarray*}
\displaystyle\int\frac{|x|}{(1-\theta x)^2}\nu(dx) & = &
\frac{1}{|\theta|}\displaystyle\int\frac{|\theta x-1+1|}{(1-\theta
x)^2}\nu(dx)\\ & \leq &
(-\frac{1}{\theta})\displaystyle\int\frac{|\theta x-1|}{(1-\theta
x)^2}\nu(dx)+
(-\frac{1}{\theta})\displaystyle\int\frac{1}{(1-\theta
x)^2}\nu(dx)\\ & \leq &
\frac{M(\theta)}{\theta}+(-\frac{1}{\theta})\frac{M(\theta)}{1-\theta
A(\nu)}<\infty.
\end{eqnarray*}
Now fix $-\infty<\alpha<\beta<\displaystyle\frac{1}{A(\nu)}$. For
$x\in supp(\nu)\subset(-\infty,0)$, the function
$$\theta\longmapsto \displaystyle\frac{\partial}{\partial
\theta}\left(\frac{1}{1-\theta x}\right)=\displaystyle\frac{x}{(1-\theta
x)^2}$$ is decreasing on $\left(-\infty, \frac{1}{A(\nu)}\right)$,
so for all $\theta\in[\alpha,\ \beta]$,
$$\displaystyle\frac{x}{(1-\beta
x)^2}\leq\displaystyle\frac{x}{(1-\theta
x)^2}\leq\displaystyle\frac{x}{(1-\alpha x)^2}.$$ We define for
$x\in supp(\nu)$
$$g(x)=\displaystyle\frac{|x|}{(1-\alpha x)^2}+\displaystyle\frac{|x|}{(1-\beta x)^2}.$$
Then $g\geq0$, and  $g$ is $\nu$-integrable, because $\alpha$ and
$\beta$ are in $\left(-\infty,\displaystyle\frac{1}{A(\nu)}\right)$, and
$\displaystyle\frac{\partial}{\partial \theta}\left(\frac{1}{1-\theta
x}\right)=\displaystyle\frac{x}{(1-\theta x)^2}\leq g(x),$ for all
$\theta\in[\alpha,\ \beta]$. Thus, one can differentiate
$M(\theta)$ under the integral sign and formula
 \eqref{L2m}
 gives
$$k'(\theta)=\frac{M(\theta)+\theta M'(\theta)-M(\theta)^2}{(\theta M(\theta))^2}.$$
The fact that
$$M(\theta)+\theta M'(\theta)-M(\theta)^2=\displaystyle\int\displaystyle\frac{1}{(1-\theta
x)^2}\nu(dx)-(\displaystyle\int\displaystyle\frac{1}{1-\theta
x}\nu(dx))^2\geq0$$ implies that the function $\theta\longmapsto
k(\theta)$ is increasing on
$\left(-\infty,\displaystyle\frac{1}{A(\nu)}\right).$
\end{proof}
We have that
\begin{eqnarray*}
\displaystyle\lim_{\theta\to-\infty} k(\theta) & =&
\displaystyle\lim_{\theta\to
-\infty}\displaystyle\frac{M(\theta)-1}{\theta M(\theta)}\\
& =& \displaystyle\lim_{\theta\to
-\infty}\displaystyle\frac{\frac{1}{\theta}G_\nu(\frac{1}{\theta})-1}{G_\nu(\frac{1}{\theta})}\\
& = & 0-\frac{1}{G_\nu(0)}=B-\frac{1}{G_\nu(B)}=m_+.
\end{eqnarray*}
For the proof of Theorem \ref{T:ext1} instead of using \eqref{m+(A)}, we define
\begin{equation}\label{M+(k)}
\mathbf{m}_+(\nu)=\displaystyle\lim_{\theta\to
\frac{1}{A(\nu)}} k(\theta).
\end{equation}
(We will later verify that  this coincides with \eqref{m+(A)} when $A(\nu)<0$.)
 Then, the function $k(.)$ realizes a
bijection from $(-\infty,\displaystyle\frac{1}{A(\nu)})$ onto its
image $(m_+,\mathbf{m}_+(\nu)).$
 We then define the function $\psi$
on $(m_0,m_+)$ as the inverse of the restriction of $k(.)$ to
$(0,\infty)$, and on $(m_+,\mathbf{m}_+(\nu))$ as
the inverse of the restriction of $k(.)$ to
$\left(-\infty,\displaystyle\frac{1}{A(\nu)}\right)$.
This leads to the parametrization by the mean
$m\in(m_0,m_+)\cup(m_+,\mathbf{m}_+(\nu))$ of the family
$\overline{\mathcal{K}}_+(\nu)$. The definition of the
pseudo-variance function can also be extended using the function
$\psi$. Following \eqref{m2v},  we define $\V(.)$ for
$m\in(m_0,m_+)\cup(m_+, \mathbf{m}_+(\nu))$ as
$$\V(m)=m\left(\frac{1}{\psi(m)}-m\right).$$
We have that $$\displaystyle\lim_{m\longrightarrow
(m_+)^-}\displaystyle\frac{1}{\psi(m)}=0=\displaystyle\lim_{m\longrightarrow
(m_+)^+}\displaystyle\frac{1}{\psi(m)},$$ so that we define $\V(.)$ at
$m_+$ by $\V(m_+)=-m_+^2$.
Note that  $Q_{m_+}(dx)=\frac{m_+}{x}\nu(dx)$ is well defined for $A(\nu)<0$.

The explicit parametrization by the means of the enlarged family
can then be given by
$$\overline{\mathcal{K}}_+(\nu)=\{Q_m(dx)=\frac{\V(m)}{\V(m)+m(m-x)}\nu(dx)
\ ;\ m\in(m_0, \mathbf{m}_+(\nu))\}.$$

 The function
$m\longmapsto\psi(m)=\displaystyle\frac{1}{\V(m)/m+m}$ is
increasing on $(m_+,\mathbf{m}_+(\nu))$, so the function
$m\longmapsto\V(m)/m+m$ is decreasing on $(m_+,\mathbf{m}_+(\nu))$
and
$$\displaystyle\lim_{m\longrightarrow\mathbf{m}_+(\nu)}\V(m)/m+m=A(\nu).$$
This implies that  \eqref{m+(A)} holds when $A(\nu)<0$.

If $A(\nu)\geq0$, then \eqref{m+(A)} gives $\mathbf{m}_+(\nu)=m_+$ because
$m_++\displaystyle\frac{\V(m_+)}{m_+}=\displaystyle\frac{1}{\theta_+}=B=A(\nu)$,
and then $\overline{\mathcal{K}}_+(\nu)=\mathcal{K}_+(\nu)$. This ends the proof of Theorem \ref{T:ext1}.


\section{The second extension}
As indicated by Examples \ref{Ex:semicircle} and  \ref{Ex:3.3},
family  $\bar{ \mathcal{K}}_+(\nu)$ may have a further extension.
Define
\begin{equation}\label{mmm+}
\mathbf{M}_+=\inf\{m>m_0: \V(m)/m<0\}.
\end{equation}
From  Remark  \ref{R1.3}(iii)  it is clear that $\mathbf{M}_+\geq
{m}_+$.  In fact,  $\mathbf{M}_+\geq \mathbf{m}_+$. This can be
seen from \eqref{m+(A)}: since the mean must be smaller than
$A(\nu)$ we have $\mathbf{m}_+ \leq A(\nu)$, so $\V(m)/m\geq 0$
for all $m<\mathbf{m}_+$.

 It is easy to see that
$\mathbf{M}_+=\infty>\mathbf{m}_+$ in Example \ref{Ex:semicircle}
and in Example \ref{Ex:3.3} while $\mathbf{M}_+=\mathbf{m}_+=m_+$ in
Example \ref{Ex:3.2}.

 We now introduce the second extension of the  family $ \mathcal{K}_+(\nu)$ as the family of measures
 $$\overline{\overline{\mathcal{K}}}_+(\nu)=\{\overline Q_m(dx): \; m_0<m<\mathbf{M}_+(\nu)\},$$
 with $\overline Q_m$ given by
\begin{equation}\label{Q++}
\overline Q_m(dx)=\frac{\V(m)}{\V(m)+m(m-x)}\nu(dx)+ p(m)
\delta_{m+\V(m)/m},
\end{equation}
 where the weight of the atom is
 $$
 p(m)=\begin{cases}
 0 & \mbox{ if $m<m_+:=B-\frac{1}{G_\nu(B)}$}\\
 1-\frac{\V(m)}{m} G_\nu\left(m+\frac{\V(m)}{m}\right) & \mbox{ if $m>m_+$ and $\V(m)/m\geq 0$}
\end{cases}
 $$
 Since   formula \eqref{G2V} holds for all $m\in(m_0,\mathbf{m}_+)$, it is clear that
 $\bar{ \mathcal{K}}_+(\nu)\subset \overline{\overline{\mathcal{K}}}_+(\nu)$.
 We now verify that the extension satisfies desired conditions.
\begin{theorem}\label{T:ext2}
 Let $\mathbf{m}_+<m<\mathbf{M}_+$. Then \eqref{Q++} defines a probability measure
$ \overline Q_m(dx)$   with mean $m$, and if $\nu$ has finite first moment  $m_0$ then the
variance of $\overline Q_m$ is
\begin{equation}\label{Q-Var-ext2}
\int (x-m)^2 \overline Q_m(dx)=\frac{(m-m_0)\V(m)}{m}.
\end{equation}

 \end{theorem}
Here the use of $\V(m)$ is based on the assumption the
pseudo-variance function $\V$ extends as a real analytic function to
$(m_0,+\infty)$. We will show later the definition of the
pseudo-variance function $\V$ may extended to
$\mathbf{m}_+<m<\mathbf{M}_+$.

Since Marchenko-Pastur law is free-infinitely divisible, from \cite[Example 4.1]{Bryc-06-08}  one can see that there is no "one simple formula"
for $m_+(\nu)$ under the free convolution power.  On the other hand, the domain of means for exponential families scales nicely under classical convolution power, and it is satisfying to note that the extended domain of means lead to the analogous formula:
\begin{equation}
\mathbf{M}_+(\nu^{\boxplus \alpha})=\alpha \mathbf{M}_+(\nu).
\end{equation}
Indeed, since  $\V_{\nu^{\boxplus
\alpha}}(m)=\alpha \V_\nu(m/\alpha)$, see
\cite[(3.17)]{Bryc-Hassairi-09},
the result follows from \eqref{mmm+}.

The rest of this section contains proof of Theorem \ref{T:ext2}.
In the proof,  we
focus on the behavior of the function
\begin{equation}\label{h}
h(m)=\displaystyle\frac{\V(m)}{m}+m, \ \textrm{for} \ m>\mathbf{m}_+(\nu),
\end{equation}
where $\mathbf{m}_+(\nu)$ is defined by \eqref{m+(A)}.
In order to
make clear the idea, we first study some examples,
\begin{example}  The Wigner's semicircle (free Gaussian) law.
$$\nu(dx)=\displaystyle\frac{\sqrt{4-x^2}}{2\pi}1_{(-2,2)}(x)dx ,$$
has a constant variance function
 $v(m)=1=\V(m)$ and the (one-sided) domain of means is
 $D_+(\nu)=(0,1),$ and $\Theta=(0,\theta_+)=(0,1/2)$.\\
$\mathbf{m}_+(\nu)=m_+=1$. We observe that, for
$m<\mathbf{m}_+(\nu)$, there exists a unique
$\overline{m}\geq\mathbf{m}_+(\nu), $ such that
$$\frac{\V(\overline{m})}{\overline{m}}+\overline{m}=\frac{\V(m)}{m}+m.$$
In fact,
$\overline{m}=\displaystyle\frac{1}{m}=\displaystyle\frac{\mathbf{m}^2_+(\nu)}{m}$.
\end{example}
\begin{example}
The (absolutely continuous) Marchenko-Pastur law
 $$\nu(dx)=\displaystyle\frac{\sqrt{4-(x-a)^2}}{2\pi(1+ax)}1_{(a-2,a+2)}(x)dx$$
 with $0<a^2<1$. The variance function is  $v(m)=1+am=\V(m)$, and
the domain of means is $D_+(\nu)=(0,1),$ and
$\Theta=(0,\theta_+)=(0,1/2)$.\\ $\mathbf{m}_+(\nu)=m_+=1$. We
also observe that for $m<\mathbf{m}_+(\nu)$, there exists a unique
$\overline{m}=\displaystyle\frac{1}{m}=\displaystyle\frac{\mathbf{m}^2_+(\nu)}{m}\geq\mathbf{m}_+(\nu),
$ such that
$$\frac{\V(\overline{m})}{\overline{m}}+\overline{m}=\frac{\V(m)}{m}+m.$$
\end{example}
\begin{example}The free strict arcsine law
$$\nu(dx)=\displaystyle\frac{\sqrt{3-4x}}{2\pi(1+x^2)}1_{(-\infty,3/4)}(x)dx $$
has pseudo-variance function $\V(m)=m(1+m^2)$, the
domain of means $ D_+(\nu)=(-\infty,-1/2)$ and $\Theta=(0,\theta_+)=(0,4/3)$.\\
$\mathbf{m}_+(\nu)=m_+=-1/2$, and for $m<\mathbf{m}_+(\nu)$, there
exists a unique
$\overline{m}=-m-1=-m+2\mathbf{m}_+(\nu)\geq\mathbf{m}_+(\nu), $
such that
$$\frac{\V(\overline{m})}{\overline{m}}+\overline{m}=\frac{\V(m)}{m}+m.$$
\end{example}
 \begin{example}[compare Example \ref{Ex:3.3}] The inverse semicircle law
$$\nu(dx)=\displaystyle\frac{p\sqrt{-p^2-4x}}{2\pi x^2}1_{(-\infty,-p^2/4)}(x)dx ,$$
has the pseudo-variance function  $\V(m)=m^3/p^2$, and the domain of
means is $D_+(\nu)=(-\infty,-p^2)$. Consider the inverse semicircle
law with $p=1$.
$\Theta=(-\infty,-4)\cup(0,+\infty)=(-\infty,1/A(\nu))\cup(0,+\infty)$.\\
$\mathbf{m}_+(\nu)=-1/2>m_+=-1$. For $m\leq\mathbf{m}_+(\nu)$,
there exists a unique
$\overline{m}=-m-1=-m+2\mathbf{m}_+(\nu)\geq\mathbf{m}_+(\nu), $
such that
$$\frac{\V(\overline{m})}{\overline{m}}+\overline{m}=\frac{\V(m)}{m}+m.$$
\end{example}
\begin{example}The free Ressel law
$$\nu(dx)=\displaystyle\frac{-1}{\pi x\sqrt{-1-x}}1_{(-\infty,-1)}(x)dx $$
has domain of means  $D_+(\nu)=(-\infty,-2)$, the
pseudo-variance function $\V(m)=m^2(m+1)$, and $\Theta=(-\infty,-1)\cup(0,+\infty)=(-\infty,1/A(\nu))\cup(0,+\infty)$.\\
$\mathbf{m}_+(\nu)=-1>m_+=-2$. For $m\leq\mathbf{m}_+(\nu)$, there
exists a unique
$\overline{m}=-m-2=-m+2\mathbf{m}_+(\nu)\geq\mathbf{m}_+(\nu), $
such that
$$\frac{\V(\overline{m})}{\overline{m}}+\overline{m}=\frac{\V(m)}{m}+m.$$
\end{example}

\subsection{Proof of Theorem \ref{T:ext2}}

Without loss of generality we suppose that $\mathbf{m}_+(\nu)<+\infty$.
\begin{definition}
For $m_1\in(m_0,\mathbf{m}_+(\nu))$, we define the set
$$\mathcal{V}_{m_1}=\{m\geq\mathbf{m}_+(\nu)\ ;\
\frac{\V(m)}{m}+m=\frac{\V(m_1)}{m_1}+m_1\}.$$
\end{definition}
 Since $\V$ is assumed analytic,  $\mathcal{V}_{m}$ is a (possibly empty) countable set with no accumulation points.

\begin{proposition}
 If for $m_1\in(m_0,\mathbf{m}_+(\nu))$,
 $\mathcal{V}_{m_1}\neq\emptyset$, then for $m$ such that $m_1\leq m\leq
 \mathbf{m}_+(\nu)$, $\mathcal{V}_{m}\neq\emptyset$.
\end{proposition}
\begin{proof}Consider the
function $h: m\longmapsto \V(m)/m+m$ and suppose that for
$m_1\in(m_0,\mathbf{m}_+(\nu))$, $\mathcal{V}_{m_1}\neq\emptyset$,
then there exists $m'_1\geq\mathbf{m}_+(\nu)$ such that
$$\V(m_1)/m_1+m_1=\V(m'_1)/m'_1+m'_1.$$
 We have that
$$h(\mathbf{m}_+(\nu))=A(\nu) \ \  \textrm{and}\ \ \
h(m'_1)=\V(m_1)/m_1+m_1.$$ For
$y\in(A(\nu),\V(m_1)/m_1+m_1)$, by continuity of $h$, there
exists $m'\in(\mathbf{m}_+(\nu),m'_1)$ such that
$y=h(m')=\V(m')/m'+m'$.\\
In other words, for all $m\in(m_1,\mathbf{m}_+(\nu))$, there
exists $m'\in(\mathbf{m}_+(\nu),m'_1)$ such that
$h(m)=y=h(m'),$ then $\mathcal{V}_{m}\neq\emptyset$.
\end{proof}
\begin{remark}  From this proposition, it follows that the set of
$m$ belonging to $(m_0,\mathbf{m}_+(\nu))$ such that
$\mathcal{V}_{m}\neq\emptyset$ is an interval.
\end{remark}

Define
$$\widetilde{\mathbf{m}}=\inf\{m\in(m_0,\mathbf{m}_+(\nu)),\ \
\textrm{such that}\ \ \mathcal{V}_{m}\neq\emptyset\}.$$ For
$m\in(\widetilde{\mathbf{m}},\mathbf{m}_+(\nu))$ let
$$\overline{m}=\inf\{\mathcal{V}_{m}\}.$$
Note that it may happen that  $\mathcal{V}_{m}=\emptyset$ for all
$m\in(m_0,\mathbf{m}_+(\nu))$, see Example \ref{Ex:3.2}.   However,  when $\widetilde{\mathbf{m}}<\mathbf{m}_+(\nu)$ we have the following.

\begin{proposition}\label{P:h}
The function
$g:m\longmapsto\overline{m}$ is (strictly) decreasing  on
 $(\widetilde{\mathbf{m}},\mathbf{m}_+(\nu))$.
\end{proposition}
\begin{proof}
Let $m_1, \ m_2\in(\widetilde{\mathbf{m}},\mathbf{m}_+(\nu))$ such
that $m_1<m_2$, the fact that the function $h$ from \eqref{h} is decreasing
on $(m_0,\mathbf{m}_+(\nu))$ implies that
$$\frac{\V(m_1)}{m_1}+m_1>\frac{\V(m_2)}{m_2}+m_2.$$
As $h(m_1)=h(\overline{m}_1)$ and
$h(m_2)=h(\overline{m}_2)$, we have that
\begin{equation}\label{(2.1)}
\frac{\V(\overline{m}_1)}{\overline{m}_1}+\overline{m}_1>\frac{\V(\overline{m}_2)}
{\overline{m}_2}+\overline{m}_2
\end{equation}
 and necessarily, we have
$\overline{m}_1>\overline{m}_2$.

Indeed, $\overline{m}_1=\overline{m}_2$ is not possible,  and if
$\overline{m}_1<\overline{m}_2$, then the inequality \eqref{(2.1)} and the
continuity of the function $h$ implies that there exists
$y<\overline{m}_1$ such that $\displaystyle\frac{\V(\overline{m}_2)}
{\overline{m}_2}+\overline{m}_2=h(y)=\displaystyle\frac{\V(m_2)}{m_2}+m_2$,
which is in contradiction with the fact that
$$\overline{m}_2=\inf\left\{m\geq\mathbf{m}_+(\nu)\ :\
\frac{\V(m)}{m}+m=\frac{\V(m_2)}{m_2}+m_2\right\}.$$

\end{proof}

We have $\overline{\mathbf{m}_+(\nu)}=\mathbf{m}_+(\nu)$, and set
$\displaystyle\widetilde{\mathbf{M}}=\lim_{m\longrightarrow\widetilde{\mathbf{m}}}\overline{m}$.

\begin{proposition}
The function
$h:\overline{m}\longmapsto\displaystyle\frac{\V(\overline{m})}
{\overline{m}}+\overline{m}$ is increasing on
$(\mathbf{m}_+(\nu),\widetilde{\mathbf{M}})$.
\end{proposition}
\begin{proof}
 Let $\overline{m}_1,\
\overline{m}_2\in(\mathbf{m}_+(\nu),\widetilde{\mathbf{M}})$ such
that $\overline{m}_1\leq\overline{m}_2.$ Then there exists $m_1,\
m_2 \in(\widetilde{\mathbf{m}},\mathbf{m}_+(\nu))$ such that
$m_1\geq m_2$ and
 $\displaystyle\frac{\V(\overline{m}_1)}{\overline{m}_1}+\overline{m}_1=\displaystyle\frac{\V(m_1)}{m_1}+m_1$
and $\displaystyle\frac{\V(\overline{m}_2)}
{\overline{m}_2}+\overline{m}_2=\displaystyle\frac{\V(m_2)}{m_2}+m_2$.
This implies that
$$\frac{\V(\overline{m}_1)}{\overline{m}_1}+\overline{m}_1\leq\frac{\V(\overline{m}_2)}
{\overline{m}_2}+\overline{m}_2.$$
\end{proof}
We are now in position to show that one can extend the definition of
the pseudo-variance function $\V(m)$ to
$m\in(\mathbf{m}_+,\mathbf{M}_+)$. We first use the mean function
$k(.)$ given in  \eqref{L2m}, to define a new mean function
$\overline{k}(.)$. In fact, we know that the function
$g:m\longmapsto\overline{m}$ realizes a bijection from
$(\mathbf{\widehat{m}},\mathbf{m}_+)$ into
$(\mathbf{m}_+,\mathbf{M}_+)$ with
$\mathbf{M}_+=g(\mathbf{\widehat{m}})$. We then define
$\overline{k}(.)$ on
$\overline{\Theta}=k^{-1}((\mathbf{\widehat{m}},\mathbf{m}_+))\subset\Theta$,
by $\overline{k}(\theta)=g(k(\theta))$. If $m=k(\theta)$, then
$\overline{m}=g(k(\theta))=\overline{k}(\theta)$.\\
We define $\overline{\psi}$ as the inverse of $\overline{k}(.)$ from
$(\mathbf{m}_+,\mathbf{M}_+)$ into $\overline{\Theta}$. The
pseudo-variance function is then defined $\forall\ \
m\in(\mathbf{m}_+,\mathbf{M}_+),$ as in \eqref{m2v}, that is
$$\V(m)=m\left(\frac{1}{\overline{\psi}(m)}-m\right),\ \ \ \forall\ \ m\in(\mathbf{m}_+,\mathbf{M}_+).$$

\subsubsection*{Conclusion of proof of Theorem \ref{T:ext2}}
Note that for $m>\mathbf{m}_+(\nu)$ formula \eqref{F(V)} defines $Q_m(dx)$ which
is not a probability measure, and   it may  be negative. Our restriction to $m< \mathbf{M}_+(\nu)$ given by  \eqref{mmm+}
guarantees its positivity, so measure $\overline{Q}_m$ is also non-negative.

We now verify that  $\overline{Q}_m$ is a probability measure with required properties for $m\in(m_0,\mathbf{M}_+(\nu))$.
We write \eqref{Q++}  explicitly
$$\overline{Q}_{m}(dx)=\frac{\V(m)}{\V(m)+m(m-x)}\nu(dx)
+(1-\frac{\V(m)}{m}G_\nu(\frac{\V(m)}{m}+m))
\delta_{\frac{\V(m)}{m}+m}.$$
Note that when
$m\in(m_0,\mathbf{m}_+(\nu))$ we have that
 $$1-\frac{\V(m)}{m}G_\nu\left(\frac{\V(m)}{m}+m\right)=0,$$ so that $\overline{Q}_m$ reduces
 to the distribution given in
\eqref{F(V)}, and has desired properties.
Therefore, without loss of generality we  restrict ourselves to  $m>\mathbf{m}_+(\nu)$ such that $\V(m)/m\geq0$.

%


From \eqref{F(V)} we have
 \begin{multline*}
\int Q_m(dx)
 =  \int \frac{\V(m)}{\V(m)+m(m-x)}\nu(dx)\\
 =  \frac{\V(m)}{m}\int\frac{1}{\frac{\V(m)}{m}+m-x}\nu(dx)\\
 =  \frac{\V(m)}{m}G_\nu\left(\frac{\V(m)}{m}+m\right).
\end{multline*}
Therefore,
%
%
\begin{multline*}
\displaystyle\int\overline{Q}_m(dx)  =
                                       \int Q_m(dx)+p(m)
                                   \\  =
          \frac{\V(m)}{m}G_\nu\left(\frac{\V(m)}{m}+m\right)+1-\frac{\V(m)}{m}G_\nu\left(\frac{\V(m)}{m}+m\right)  \\
 =  1\ ,\ \forall m\in(m_0,\textbf{M}_+).
\end{multline*}

We now verify that $\displaystyle\int \overline{Q}_m(dx)=m$.
We have
\begin{eqnarray*}
\displaystyle\int x\overline{Q}_m(dx)
& = & \int x Q_m(dx)+\int x p(m)\delta_{\frac{\V(m)}{m}+m}(x)dx\\
& = & \int
\frac{\V(m)x}{\V(m)+m(m-x)}\nu(dx)+\left(\frac{\V(m)}{m}+m\right)p(m).
\end{eqnarray*}
The integral is
\begin{eqnarray*}
\int \frac{\V(m)x}{\V(m)+m(m-x)}\nu(dx)
& = & \frac{-\V(m)}{m}\int\frac{-x}{\V(m)/m+m-x}\nu(dx)\\
& = & \frac{-\V(m)}{m}\int\frac{(\V(m)/m+m)-x-(\V(m)/m+m)}{\V(m)/m+m-x}\nu(dx)\\
& = & \frac{-\V(m)}{m}\left[1-(\V(m)/m+m)G_\nu(\frac{\V(m)}{m}+m)\right]\\
& = & \frac{-\V(m)}{m}\left[1-\frac{\V(m)}{m}G_\nu(\frac{\V(m)}{m}+m)-m G_\nu(\frac{\V(m)}{m}+m)\right]\\
& = &
\frac{-\V(m)}{m}\left[1-\frac{\V(m)}{m}G_\nu(\frac{\V(m)}{m}+m)\right]\\
 & + &
\V(m)G_\nu(\frac{\V(m)}{m}+m).
\end{eqnarray*}
On the other hand we have\\
 \begin{eqnarray*}
 \left(\frac{\V(m)}{m}+m\right)p(m)
& = &\left(\frac{\V(m)}{m}+m\right)\left[1-\frac{\V(m)}{m}G_\nu\left(\frac{\V(m)}{m}+m\right)\right]\\
 & = & m-\V(m)G_\nu\left(\frac{\V(m)}{m}+m\right)+\frac{\V(m)}{m}
\left[1-\frac{\V(m)}{m}G_\nu\left(\frac{\V(m)}{m}+m\right)\right].
\end{eqnarray*}
Thus
$$\int x\overline{Q}_m(dx)=\int x Q_m(dx)+\int x p(m)\delta_{\frac{\V(m)}{m}+m}(x)dx=m. $$

Next we verify formula \eqref{Q-Var-ext2}.
We have
\begin{eqnarray*}
\int x(x-m)\overline{Q}_m(dx)
 & = &  \int x(x-m)Q_m(dx)+\int x(x-m)p(m)\delta_{m+\frac{\V(m)}{m}}\\
 & = & \int x(x-m)Q_m(dx)+\left(m+\frac{\V(m)}{m}\right)\frac{\V(m)}{m}p(m).
\end{eqnarray*}
The integral is
\begin{eqnarray*}
\int x(x-m)Q_m(dx)
& = & \frac{\V(m)}{m}\left[\int x Q_m(dx)-\int x\nu(dx)\right]\\
& = &
\frac{\V(m)}{m}\Big[\left(\frac{-\V(m)}{m}\right)\left(1-\frac{\V(m)}{m}G_\nu\left(\frac{\V(m)}{m}+m\right)\right)\\
& + & \V(m)G_\nu\left(\frac{\V(m)}{m}+m\right)-m_0\Big].
\end{eqnarray*}
On the other hand,
$$\left(m+\frac{\V(m)}{m}\right)\frac{\V(m)}{m}p(m)=\left(m+\frac{\V(m)}{m}\right)\frac{\V(m)}{m}
\left[1-\frac{\V(m)}{m}G_\nu\left(\frac{\V(m)}{m}+m\right)\right] $$
$$=\frac{\V(m)}{m}\left[m-\V(m)G_\nu\left(\frac{\V(m)}{m}+m\right)
+\frac{\V(m)}{m}\left(1-\frac{\V(m)}{m}G_\nu\left(\frac{\V(m)}{m}+m\right)\right)\right].$$
Thus
$$\int x(x-m)\overline{Q}_m(dx)=\frac{\V(m)}{m}\left(m-m_0\right).$$


\end{document}